\documentclass[12pt,letterpaper]{amsart}

\usepackage[utf8]{inputenc}

\usepackage[all]{xypic}

\theoremstyle{definition}
\newtheorem{df}{Definition}
\theoremstyle{plain}
\newtheorem{stat}{Statement}

\newtheorem{thm}{Theorem}

\newtheorem{cor}{Corollary}
\theoremstyle{remark}
\newtheorem{ex}{Example}
\newtheorem{rem}{Remark}

\def\rest#1#2{\left. #1 \right|_{#2}}
\author{F. A. Arias and M. Malakhaltsev}
\address{Universidad de los Andes, Bogotá, Colombia}
\email{fa.arias44@uniandes.edu.co; mikarm@uniandes.edu.co}

\title{A generalization of the Gauss-Bonnet and Hopf-Poincaré theorems. Part II}

\thanks{This investigation was supported by Vicerrectoría de Investigaciones and the Faculty of Sciences of Universidad de los Andes}

\begin{document}

\begin{abstract}
This paper is a continuation of \cite{arias_malakhaltsev2015}.
Let $\pi : E \to M$ be a locally trivial fiber bundle over a two-dimensional manifold $M$, and $\Sigma \subset M$ be a discrete subset.
A subset $Q \subset E$ is called a \emph{$n$-sheeted branched section of the bundle $\pi$} if $Q' = \pi^{-1}(M \setminus \Sigma) \cap Q$ is a $n$-sheeted covering of $M \setminus \Sigma$.
The set $\Sigma$ is called the \emph{singularity set} of the branched section $Q$.
We define the index of a singularity point of a branched section, and give examples of its calculation, in particular for branched sections of the projective tangent bundle of $M$ determined by binary differential equations.
Also we define a resolution of singularities of a branched section, and prove an analog of Hopf-Poincar\'e-Gauss-Bonnet theorem for the branched sections admitting a resolution.
\end{abstract}

\maketitle



\section{Introduction}
\label{sec:1}

Let us recall that a branched covering is a smooth map $f : X \to Y$, where $X$ and $Y$ are compact $n$-dimensional manifolds, such that $df_x : T_x X \to T_{f(x)}Y$ is an isomorphism for all points $x \in X \setminus A$ for some subset $A \subset X$ of dimension less or equal to $n - 2$.
In this case, if $X' = X \setminus f^{-1}(f(A))$ and $Y' = Y \setminus f(A)$, then the induced map $f' : X' \to Y'$ is a finite-sheeted covering map.
The points of the set $f(A)$ are called the \emph{branch points} of the branch covering $f$
(\cite{dubrovin_novikov_fomenko1985}, Section~18.3).

Now let $\xi = \left\{ \pi_E : E \to M \right\}$ be a fiber bundle.
Let $\Sigma$ be a closed subset of $M$, $M' = M \setminus \Sigma$, and $E' = \pi^{-1}(M')$.
\begin{df}
An \emph{$n$-sheeted branched section} of the bundle $\xi$ is a subset $Q \subset E$ such that $Q' = Q \cap E'$ is an embedded submanifold of $E$ and $\rest{\pi_E}{Q'} : Q' \to M'$ is a $n$-sheeted covering.
The set $\Sigma$ is called the \emph{singularity set} of the branched section $Q$.
\label{df:1}
\end{df}
\begin{ex}
Let $V$ be a section of the tangent bundle $\pi_{TN} : TN \to N$, and $f : N \to M$ be a $k$-sheeted covering, then 
we can construct a branched section $df(V)$ of the tangent bundle  $\pi_{TM} : TM \to M$ in the following way. 
Let us consider the subset $Q = \left\{ df_y(V(y)) \mid y \in N \right\} \subset TM$.
For each $x \in M$, let us set $\mathcal{V}(x) = \left\{ df_y(V(y)) \mid y \in f^{-1}(x) \right\} \subset T_x M$.
Take the subset $\Sigma \subset M$ consisting of points $x \in M$ such that the number of elements of the set $\mathcal{V}(x)$ is less than $k$.
Then $M' = M \setminus \Sigma$ is open, $Q' = Q \cap \pi^{-1}_{TM}(M')$ is a submanifold of $TM$ and $f$ induces a $k$-sheeted covering $f' : Q' \to M'$.
Indeed, for each $x \in M'$ there exists a neighborhood $U \subset M'$ of $x$ such that $f^{-1}(U) = \mathop\sqcup\limits_{j=1}^k \widetilde{U}_j \subset N$ and, for each $j = \overline{1, k}$, the application $f_j = \rest{f}{\widetilde{U}_j} : \widetilde{U}_j \to U$ is a diffeomorphism.
Therefore $df_j : T\widetilde{U}_j \subset TN \to TU \subset TM$ is also a diffeomorphism. 
As $V : \widetilde{U}_j \to V(\widetilde{U}_j) \subset TN$ is a diffeomorphism onto its image,
the map $\theta_j = df_j \circ V \circ f_j^{-1} : U \to df_j (V(\widetilde{U}_j)) \subset Q \subset TM$, $j = \overline{1, k}$, is a diffeomorphism onto its image, as well. 
Note that, for each $y \in U \subset M'$, we have that the set $f^{-1}(y) = \left\{ p_j \in \widetilde{U}_j \right\}_{j = \overline{1, k}}$ consists of $k$ distinct points, and the set $\left\{ df_{p_j}(V(p_j)) \right\}_{j = \overline{1, k}}$ consists of $k$ distinct vectors, by the definition of $M'$.  Therefore, $\theta_i (U) \cap \theta_j(U) = \emptyset$, for $i \ne j$. 
Thus $\pi_{TM}^{-1}(U) \cap Q' = \mathop\sqcup\limits_{j=1}^k \theta_j(U)$, this means that $U$ is simply covered in $Q'$, and $Q'$ is a $k$-sheeted covering of $M'$.
\label{ex:1}
\end{ex}

The branched sections naturally appear in the theory of differential equations over manifolds.
Our main example in this paper is the following one.
\begin{ex}
Let $M$ be a connected compact oriented manifold and let $\omega$ be a symmetric tensor of order $n$ over $M$. Recall that such a tensor can be written locally as follows
\begin{equation}
\omega_{(x, y)}=a_0(x,y)dx^n+a_1(x, y)dx^{n-1}dy+\cdots +a_n(x, y)dy^n,
\label{eq:22}
\end{equation}
where $(x, y)$ are coordinate functions on an open set $U\subset M$, and $a_i: U\to \mathbb{R}$ are smooth functions defined in $U$.
In what follows, we suppose that $\omega$ has the following properties:
\begin{enumerate}
\item The function $\omega_{(x,y)}$ is identically zero if and only if $a_i(x, y)=0$ for $0\leq i\leq n$. We set $\Sigma=\{p\in M : \omega_p=0\}$.
\item On $M\setminus \Sigma$, the tensor $\omega$ has the form $\omega=\lambda_1\lambda_2\cdots\lambda_n$, where $\lambda_i\in \Omega(M\setminus \Sigma)$ pairwise linearly independent.
\end{enumerate}
\begin{stat}
The $n$-form $\omega$ determines a branched section of the bundle $\pi: PTM\to M$
\label{stat:6}
\end{stat}
\begin{proof}
Let $Q$ be the solution on $PTM$ of the equation \eqref{eq:22}.
We will prove that $Q$ is a branched section of $\pi$.
Let $E'=\pi^{-1}(M\setminus \Sigma)$ and $Q'=Q\cap E'$.
It follows from the property (2) that the set $F_p=Q\cap \pi^{-1}(p), \ p\in M\setminus \Sigma$ has exactly $n$ elements, therefore each fiber of the surjective map $\pi':=\pi|_{Q'}: Q'\to M\setminus \Sigma$ is finite with $n$ elements.
On the other hand, if $\varphi: \pi^{-1}(U)\to U\times \mathbb{R}P^1$ is a trivialization of $PTM$ on $U$, then the restriction $\varphi':=\varphi|_{\pi^{-1}(U)\cap Q'}: \pi^{-1}(U)\cap Q' \to U\times \mathbb{R}P^1$ is a homeomorphism on its image.
Since $\pi|_{\pi^{-1}(U)\cap Q'}:\pi^{-1}(U)\cap Q'\to U\cap (M\setminus \Sigma)$ has finite fiber with $n$ elements over each point $p\in U\cap (M\setminus \Sigma)$, from the following commutative diagram
\begin{equation}
\xymatrix{
&\varphi'(\pi ^{-1} (U\cap (M\setminus \Sigma))\cap Q') \ar[ld]_{pr_1} & \\
U\cap (M\setminus \Sigma) & \pi ^{-1} (U\cap (M\setminus \Sigma))\cap Q'. \ar[l]_{\pi'} \ar[u]_{\varphi'}}
\label{eq:23}
\end{equation}
It follows that $\pi|_{Q'}: Q'\to M\setminus \Sigma$ is a local diffeomorphism.
Therefore, $\pi|_{Q'}: Q'\to M\setminus \Sigma$ is a $n$-sheeted branched covering, and so $Q$ is a branched section of $PTM$.
\end{proof}

Geometrically $Q$ determines an $n$-web at the points of $M \setminus \Sigma$.
\label{ex:2}
\end{ex}

\begin{ex}
Let $\xi = \left\{ \pi : \overline{P} \to M \right\}$ be a $\overline{G}$-principal bundle which reduces to a finite subgroup $G \subset \overline{G}$ over $M \setminus \Sigma$, where $\Sigma \subset M$ is a closed subset.
Then the corresponding $G$-principal bundle $P \subset \overline{P}$ is a branched section of the bundle $\xi$ with singularity set $\Sigma$. 

For example, let $M$ be a two-dimensional oriented Riemannian manifold, and $\overline{P} = SO(M)$, the $SO(2)$-principal bundle of orthonormal positively oriented frames of $M$.
Any finite subgroup $G \subset SO(2)$ is a cyclic group $G \cong \mathbb{Z}_m$ generated by the rotation $R_{2\pi/m}$. 

If $P \subset SO(M') \subset SO(M)$ is a reduction of $SO(M)$ to $G$ over $M' = M \setminus \Sigma$, then at each point $x \in M'$ we have the set $\mathcal{N}(x) = \left\{ e \in T_x M \mid (e, R_{\pi/2}e) \in P  \right\}$, which consists of $m$ unit vectors such that the angle between any two of them is $2\pi l/m$.
The set $\mathcal{N}(x)$ defines a regular $m$-polygon $P_m \subset T_x M'$ inscribed in a unit circle centered at $0 \in T_x M$.

It is clear that, vice versa, if at any point of $M' = M \setminus \Sigma$, we are given a unitary $m$-polygon $P_m \subset T_p M'$ and the field of these polygons is smooth (these means that locally we can choose $m$ unitary vector fields whose values are  the vertices of the polygons $P_m$), then the bundle $SO(M)$ reduces to the subgroup $G \cong \mathbb{Z}_m$ of the Lie group $SO(2)$.

This situation occurs, for example, when $M$ is a surface in $\mathbb{R}^3$, and $\Sigma$ is the set of umbilic points of $M$.
Then at each point of $M'$ we have two orthogonal eigenspaces of the shape operator of the surface,  
which determine a square in $T_p M$ with vertices at points where these eigenspaces meet the unit circle centered at $0 \in T_p M$.  
Therefore, over $M' = M \setminus \Sigma$ the bundle $SO(M)$ reduces to the subgroup $G \cong \mathbb{Z}_4$ generated by the rotation $R_{\pi/2}$.
The corresponding principal subbundle $P$, the branched section of the bundle $SO(M) \to M$, consists of oriented orthogonal frames such that the frame vectors span the eigenspaces.

Moreover, as the difference of the principal curvatures never vanish on $M'$, we can order the principal curvatures in such a way that $k_1(p) > k_2(p)$ at any $p \in M'$. 
Let $L_a(p)$, $a = 1, 2$ be the eigenspace corresponding to the principal curvature $k_a(p)$, $a = 1, 2$. 
Then we can choose the subbundle $P \subset SO(M)$ in such a way that, for $\left\{ e_1, e_2 \right\} \in P$, the vector $e_a$ spans $L_a$, $a = 1, 2$, therefore in this case the bundle $SO(M) \to M$ reduces to the group $G \cong \mathbb{Z}_2$.

Also, note that this example is related to Example~\ref{ex:2}.
Indeed if we have the reduction of $P \subset SO(M)$ to the subgroup $G \cong \mathbb{Z}_m$ over $M'$, then at each point $p \in M'$ we have $m$ (or $m/2$) subspaces spanned by the vector $e_1$ from the frame $\left\{ e_1, e_2 \right\} \in P$.
Then we can take the binary differential equation~\eqref{eq:22} such that these subspaces are the roots of the corresponding algebraic equation.
\label{ex:3}
\end{ex}

The paper is organized as follows. In Section~\ref{sec:2} we define the index of an isolated singular point of a branched section of locally trivial bundle $\xi = \left\{ \pi_E : E \to M \right\}$ over a two-dimensional oriented manifold $M$ (see Definition~\ref{df:2}), this definition generalizes the definition of the index of a singular point of a section from (\cite{arias_malakhaltsev2015}, Section~2.2, Definition~1).
In Section~\ref{sec:3} we define a resolution of a branched section (see Definition~\ref{df:3}), and give various examples of resolutions (see Examples~\ref{ex:5}~--~\ref{ex:8}).
And, finally, in Section~\ref{sec:3} we prove an analogue of the Gauss-Bonnet theorem for a branched section which admits resolution (see Theorem~\ref{thm:1}).

\section{The index of a singular isolated point}
\label{sec:2}

\subsection{Local monodromy group}
\label{subsec:2_1}
Let $M$ be a two-dimensional closed oriented manifold.
Let $\xi = \left\{ \pi_E : E \to M \right\}$ be a fiber bundle with oriented typical fiber $F$.

Let us consider a $k$-sheeted branched section $Q$ of $\xi$ (see Definition~\ref{df:1}) with singularity set $\Sigma$, and let $\pi_Q = \rest{\pi_E}{Q} : Q \to M$.
Recall that $M' = M \setminus \Sigma$, $E' = \pi^{-1}(M')$, and $Q' = Q \cap E'$.

Assume that $x \in \Sigma$ is an isolated point of $\Sigma$.
Let us take a neighborhood $U(x)$ such that $U'(x) = U(x) \setminus \left\{ x \right\}$ is an open subset of $M'$ and
there exists a diffeomorphism $\varphi : (D, 0) \to (U(x), x)$, where $D \subset \mathbb{R}^2$ is the standard open $2$-disk centered at the origin $0 \in \mathbb{R}^2$.
We will call $U(x)$ a \emph{disk neighborhood} of $x$ and assume that $\varphi$ sends the standard orientation of $D$ to the orientation of $U(x)$ induced by the orientation of $M$.

By Definition~\ref{df:1}, the map $\rest{\pi_Q}{\pi_Q^{-1}(U'(x))} : \pi_Q^{-1}(U'(x)) \to U'(x)$ is a $k$-sheeted covering.

If $U(x)$ is a disk neighborhood of an isolated point $x \in \Sigma$, then for each point $y \in U'(x)$, the fundamental group $\Pi_1(y) = \pi_1(U'(x), y)$ is isomorphic to $\mathbb{Z}$.
There are two generators of $\Pi_1(y)$: $[\gamma_+]$ and $[\gamma_-]$, where $\gamma_{\pm} = \phi(C_{\pm})$ and $C_{\pm}$ is a circle in $D$ passing through the point $\varphi^{-1}(y)$ and enclosing the origin, and having positive (negative, respectively) orientation.
We will call the element $[\gamma_{\pm}] \in \Pi_1(y)$ the positive (the negative, respectively) generator of $\Pi_1(y)$.

The group $\Pi_1(y)$ acts on the fiber $Q_y = \pi_Q^{-1}(y)$ in the following way: for an element $[\gamma] \in \Pi_1(y)$ and $q \in Q_y$ we set $[\gamma]\cdot q = \bar q$ if the lift $\widetilde{\gamma}$ of $\gamma$ starting at $q$ terminates in $\bar q$.
This action is well defined, this means that if $\gamma_1$ and $\gamma_2$ represent the same element in $\Pi_1(y)$, then the lifts $\widetilde{\gamma}_1$ and $\widetilde{\gamma}_2$ starting at a same point $q$ terminate at a same point $\bar q$.

This action is a homomorphism of the group $\Pi_1(y)$ to the group of permutations of the fiber $Q_y$ and its image
is called the \emph{local monodromy group} of the branched section $Q$ at the point $y \in M'$.

\begin{stat}
The local monodromy group does not depend on a choice of the disk neighborhood $U(x)$.
\label{stat:1}
\end{stat}
\begin{proof}
Let $U(x)$ and $V(x)$ be two disk neighborhoods of $x$, and $y$ lies in $U(x) \cap V(x)$.
Then $\Pi^U_1(y) = \pi_1(U'(x), y) = \Pi^V_1(y) = \pi_1(V'(x), y)$ because for each class $[\gamma] \in \pi_1(V'(x), y)$ or $[\gamma] \in \pi_1(U'(x), y)$ one can find a representative $\gamma_1 \in [\gamma]$ which takes values in $U'(x) \cap V'(x)$.
\end{proof}

\begin{stat}
Let $\gamma$ be a loop in $U'(x)$ based at a point $y \in U'(x)$ such that its homotopic class represents the positive generator of $\Pi_1(y)$.
Then for each orbit $O$ of the local monodromy group action on $Q_y$ and each point $q \in O$, there exists a loop $\widetilde{\gamma}$ in $\pi_Q^{-1}(U'(x))$ based at $q$ which passes through each point of the orbit once and only once and such that $\pi_1(\pi_E) \left( [\widetilde{\gamma}] \right) = [\gamma]^k$, where $k$  is the number $\# O$ of elements of the orbit $O$.
Here $\pi_1(\pi_E) : \pi_1(\pi_Q^{-1}(U'(x)), q) \to \pi_1(U(x), y)$ is the homomorphism of the fundamental groups induced by the map $\pi_E$. 
\label{stat:2}
\end{stat}
\begin{proof}
First of all note that if we have an action of the group $\mathbb{Z}$ on a finite set, then we can enumerate elements of each orbit in such a way that the action of the group generator $1$ on this orbit is represented by the cycle $\sigma = (2, 3, \cdots, 1)$.
Indeed, let $O$ be an orbit of the action, and $q \in O$.
The map $F : \mathbb{Z}/H_q \to O$, $[g] \to g \cdot q$, where $H_q$ is the isotropy subgroup of the action, is an equivariant bijection.
The group $H_q$ is a cyclic group, this means that there exists $k \in \mathbb{Z}$, $k \ge 0$ such that $H_q = \left\{ k m \mid m \in \mathbb{Z} \right\}$, therefore $\mathbb{Z}/H_q = \left\{ [0], [1], \cdots, [k-1] \right\}$, and the action of the generator $1 \in \mathbb{Z}$ on $\mathbb{Z}/H_p$ is given exactly by the cycle $\sigma$.

Now, for a point $y \in U'(x)$, let $[\gamma]$, $\gamma : [0, 1] \to U'(x)$, be the positive generator of $\Pi_1(y)$.
Let us take an orbit $O$ of the local monodromy group action on $Q_y$ and a point $q \in O$.
Let $k$ be the number of elements of $O$.
As we have seen, the action of $[\gamma]$ on $O$ is represented by the cycle $\sigma$, this means we can enumerate the points of the orbit $O$ in such a way that $q_1 = q$, $[\gamma]q_1 = q_2$, \ldots, $[\gamma]q_{k-1} = q_k$, and $[\gamma]q_k = q_1 = q$.
Therefore, by the construction of the action of $\Pi_1(y)$ on $Q_y$, for the lift $\widetilde{\gamma}_1$ of $\gamma$ to $Q'$ such that $\widetilde{\gamma}_1(0) = q_1$ we have that $\widetilde{\gamma}_1(1)=q_2$, for the lift $\widetilde{\gamma}_2$ of $\gamma$ to $Q'$ such that $\widetilde{\gamma}_2(0) = q_2$ we have that $\widetilde{\gamma}_1(1)=q_3$, \ldots, and finally for the lift $\widetilde{\gamma}_k$ of $\gamma$ to $Q'$ such that $\widetilde{\gamma}_k(0) = q_k$ we have that $\widetilde{\gamma}_k(1)=q_k = q$.

What do we do in fact is that we take a point $q_1 = q \in Q_y$, then construct the points $q_2 = [\gamma] q_1$, $q_3 = [\gamma] q_2$, \ldots, up to $[\gamma]q_{k} = q_1$. 
Then the set $\left\{ q_1, q_2, \cdots, q_k \right\}$ is the orbit $O$ of the point $q$.  

It is clear that $\widetilde{\gamma} = \widetilde{\gamma}_k \cdot \widetilde{\gamma}_{k-1} \cdot \cdots \cdot\widetilde{\gamma}_1$, where ${\cdot}$ is the path composition, is a loop in $\pi_Q^{-1}(U'(x))$ at the point $q_1=q$, $\widetilde{\gamma}$ passes once and only once through each point of $O$, and $\pi_1(\pi_E) \left( [\widetilde{\gamma}] \right) = [\gamma]^k$. Thus $\widetilde{\gamma}$ is the required loop.
\end{proof}

\subsection{The index of isolated singular point}
\label{subsec:2_2}

Let $M$ be an oriented two-dimensional manifold.
Let $\xi = \left\{ \pi_E : E \to M \right\}$ be a locally trivial fiber bundle with standard fiber $F$ and a connected structure Lie group $G$.

Let $Q$ be a branched section of $\xi$ with singularity set $\Sigma$, and $x$ be an isolated point of $\Sigma$.
Take a disk neighborhood $U(x)$, and for a point $y \in U'(x)$, let $\mathcal{O}_y$ be the set of orbits of local monodromy group action on $Q_y$.
Take an orbit $O \in \mathcal{O}(y)$ and a point $q \in O$.
Let $[\gamma]$ be a positive generator of the group $\Pi_1(y)$, and $\widetilde{\gamma}$ the loop at $q$ constructed in Statement~\ref{stat:2}.

Let $\psi : \pi_E^{-1}(U(x)) \to U(x) \times F$ be a trivialization of the bundle $\xi$, and $p : \pi_E^{-1}(U(x)) \to F$ be the corresponding projection.
Then the element $[p \circ \widetilde\gamma] \in \pi_1(F)$ is called the \emph{index of the branched section $Q$ at the singular point $x$ corresponding to the orbit $O \in \mathcal{O}_y$}, call it $ind_x(Q; y, O)$.

\begin{stat}
$\hphantom{xxx}$
\begin{enumerate}
\item[a)]
The index $ind_x(Q; y, O)$ does not depend on a choice of the loop $\gamma : [0, 1] \to U(x')$ representing the positive generator of the group $\Pi_1(y)$.
\item[b)]
The index $ind_x(Q; y, O)$ does not depend on a trivialization.
\item[c)]
The index $ind_x(Q; y, O)$ does not depend on a choice of the disk neighborhood $U(x)$, this means that, if $U(x)$ and $V(x)$ are two disk neighborhoods of $x$, and $y \in U(x) \cap V(x)$, then the constructions of $ind_x(Q; y, O)$ performed for $U(x)$ and for $V(x)$ result in the same element in $\pi_1(F)$.  
\end{enumerate}
\label{stat:3}
\end{stat}
\begin{proof}
a) If $\gamma$ and $\mu$ are two representatives  of the positive generator of $\Pi_1(y)$, then $\gamma$ is homotopic to $\mu$, therefore $\gamma^k$ is homotopic to $\mu^k$, therefore the lift $\widetilde{\mu}$ of $\mu^k$ is homotopic to the lift $\widetilde{\gamma}$ of $\gamma^k$, hence $p\widetilde{\gamma}$ is homotopic to $p\widetilde{\mu}$.

b) This is because the gluing functions are homotopic to the identity as the structure group is connected.

c) This follows directly from the fact that $\Pi^U_1(y) = \Pi^V_1(y)$ (see the proof of Statement~\ref{stat:2}), and from a).
\end{proof}

\begin{ex}
Let us consider the trivial bundle $\xi = \left\{ \pi_E : E=\mathbb{C} \times \mathbb{C}^* \to M=\mathbb{C} \right\}$, where $\mathbb{C}^* = \mathbb{C} \setminus \left\{ 0 \right\}$ and $\pi_E(z, w) = z$. 
Let us take the subset $Q = \left\{ (z, w) \mid z^2 = w^3   \right\} \subset E$.

As $\rest{\pi_E}{Q}: Q \to M \setminus \left\{ z = 0 \right\}$ is a $3$-sheeted covering, we see that $Q$ is a $3$-sheeted branched section of the bundle $\xi$. 

It is clear that the singularity set of $Q$ is $\Sigma = \left\{ 0 \right\}$, so $Q$ has only one singular point $z = 0$ and this point is isolated. 
For the disk neighborhood of the isolated singular point $z = 0$ we take the entire $M = \mathbb{C}$.

Let us take $y = 1$, then $Q_y = \left\{ a=(1,1), b=(1, \varepsilon), c=(1, \varepsilon^2) \right\}$, where $\varepsilon = \exp(2\pi i/3)$.
The loop $\gamma(t) = \exp(2\pi i t)$, $t \in [0, 1]$, represents the positive generator of the group $\Pi(y=1)$, and
the lift $\widetilde{\gamma}_a$ of $\gamma$ which starts at the point $a = (1, 1)$ is given by $\widetilde{\gamma}_a(t) = (\exp(2\pi i t), \exp(\frac{4}{3} \pi i t)$, $t \in [0, 1]$.
Therefore $[\gamma]a = c$.
In the same manner one can prove that $[\gamma]b = a$, $[\gamma]c = b$.

Thus, the orbit $O$ of the point $a$ is $Q_{y=1} = \left\{ a, b, c \right\}$, and
for a representative of the class $[\widetilde{\gamma}]$ constructed in Statement~\ref{stat:2} we can take the loop $\widetilde{\gamma} = (\exp(6\pi i t), \exp(4 \pi i t))$ for $t \in [0, 1]$.

Therefore, the loop $p \widetilde{\gamma} : [0, 1] \to \mathbb{C}^*$ is given by
$p \widetilde{\gamma} = \exp(4 \pi i t))$ for $t \in [0, 1]$.
Hence
\begin{equation}
ind_0(Q; y=1, Q_{y=1}) = 2 \in \mathbb{Z} = \pi_1(\mathbb{C}^*).
\label{eq:3}
\end{equation}
\label{ex:4}
\end{ex}

Let us consider the finite set of elements of $\pi_1(F)$:
\begin{equation}
ind_x(Q; y) = \left\{ ind_x(Q; y, O) \mid O \in \mathcal{O}(y) \right\}.
\label{eq:4}
\end{equation}

\begin{stat}
The set $ind_x(Q; y)$ does not depend on $y \in U'(x)$.
\label{stat:4}
\end{stat}
\begin{proof}
Let $y, \bar y$ be two points in $U'(x)$.
Take a curve $\delta : [0, 1] \to U'(x)$ such that $\delta(0) = y$, $\delta(1) = \bar y$.

The curve $\delta$ defines the group isomorphism $\psi_\delta : \Pi_1(y) \to \Pi_1(\bar y)$, $[\gamma] \mapsto [\delta^{-1} \cdot \gamma \cdot \delta]$, where $\delta^{-1}(t) = \delta(1 - t)$.
Also, $\delta$ defines the bijection $\widetilde{\psi}_\delta : Q_y \to Q_{\bar y}$,
$q \in Q_y \mapsto \bar q \in Q_{\bar y}$, such that for the lift $\widetilde{\delta}$ of $\delta$ to $Q$ with $\widetilde{\delta}(0) = q$ we have that $\widetilde{\delta}(1) = \bar q$.
In addition, the bijection $\widetilde{\psi}_\delta$ is equivariant is sense that
$\widetilde{\psi}_\delta([\gamma]q) = \psi_\delta([\gamma])\widetilde{\psi}_\delta(q)$.

Therefore $\widetilde{\delta}$ induces a bijection $\alpha_\delta : \mathcal{O}(y) \to \mathcal{O}(\bar y)$, $O_q \mapsto O_{\widetilde{\psi}_\delta(q)}$, where $O_q$ is the $\Pi_1(y)$-orbit of the point $q \in Q_y$ and $O_{\widetilde{\psi}_\delta(q)}$ is the $\Pi_1(\bar y)$-orbit of the point  $\widetilde{\psi}_\delta(q) \in Q_{\bar y}$.

Let us prove that the loop $\widetilde{\gamma}$ which passes through the points of the orbit $O_q \in \mathcal{O}(y)$ constructed in Statement~\ref{stat:2} is homotopic in $\pi^{-1}(U'(x))$ to the corresponding loop of the orbit $O_{\widetilde{\psi}(q)} \in \mathcal{O}(\bar y)$.

Let $\gamma$ be a loop at $y \in U'(x)$ which represents the positive generator of $\Pi_1(y)$.
The loop $\widetilde{\gamma}$ constructed in Statement~\ref{stat:2} is homotopic to the lift of the loop $\gamma^k$ starting at a point $q \in Q_y$.
As the loop $\gamma^k$ is freely homotopic to the loop ${\bar \gamma}^k$, where $\bar\gamma = \delta^{-1}\gamma\delta$,
the lift $\widetilde{\gamma}$ is freely homotopic to the lift of ${\bar \gamma}^k$ starting at the point $\widetilde{\psi}_\delta(q)$, but this lift is in turn homotopic to the loop $\widetilde{\bar\gamma}$.

Therefore the loops $p \widetilde{\gamma}$ and $p \widetilde{\bar\gamma}$ are freely homotopic in $F$, therefore define the same element in $\pi_1(F)$.
Thus we have that $ind_x(Q; y, O) = ind_x(Q; \bar y, \alpha_\delta(O))$ for all $O \in \mathcal{O}(y)$, hence $ind_x(Q; y) = ind_x(Q; \bar y)$.
\end{proof}

\begin{cor}
The set $ind_x(Q)$ does not depend on the disk neighborhood $U(x)$, this means if $U_1(x)$ and $U_2(x)$ are disk neighborhoods of an isolated singular point $x \in \Sigma$, and $y_1 \in U'_1(x)$ and $y_2 \in U'_2(x)$, then the set $ind_x(Q; y_1)$ constructed via $U_1(x)$ and the set $ind_x(Q; y_2)$ constructed via $U_2(x)$ coincide.
\label{cor:1}
\end{cor}
\begin{proof}
Follows from Statement~\ref{stat:5}
\end{proof}

From Statement~\ref{stat:4} it follows that we can give the following definition.
\begin{df}
Let $Q$ be a branched section of the bundle $\xi$.
The \emph{index of $Q$ at $x \in M$} is
\begin{equation}
ind_x(Q) = ind_x(Q; y),
\label{eq:5}
\end{equation}
where $y$ is a point of $U'(x)$, where $U(x)$ is a disk neighborhood of $x$.
\label{df:2}
\end{df}

Let us fix an element $a \in H^1(F)$.
The index of $Q$ at a point $x$ with respect to $a$ is
\begin{equation}
ind_x(Q; a) = \sum_{ O \in O(y)} \frac{1}{\# O}\langle a, ind_x(Q; y, O)\rangle
= \sum_{ O \in O(y)} \frac{1}{\# O} \int_{\gamma(Q; y, O)}\alpha,
\label{eq:6}
\end{equation}
where $\alpha \in \Omega^1(F)$ represents $a \in H^1(F)$ and $\gamma(Q; y, O)$ represents the class $ind_x(Q; y, O) \in \pi_1(F)$.

\begin{ex}
Let $M$ be a connected compact oriented manifold and let $\omega$ be a symmetric tensor of order $n$ over $M$. 
In Example~\ref{ex:2} we have constructed a branched section $Q \subset PTM$ determined by the  binary differential equation~\eqref{eq:22}. 

If we consider the covering $q: \mathbb{S}^1TM \to PTM$ given by $q((p, \vec{v}))=[\vec{v}]$, we see that $q\circ \pi : \mathbb{S}^1TM \to M$ is a fiber bundle and $q^{-1}(Q)$ is a $2n$-sheeted branched covering of the bundle $\mathbb{S}^1TM\to M$.
Let $p\in \Sigma$ be a singular point, $U'(p)$ be a neighborhood disk of $p$, and $\mathcal{O}_p=\{O_1, \cdots, O_r\}$ the set of the orbits of the action of $\pi_1(U'(p))$ on $\pi^{-1}(p)$.
From the equation \eqref{eq:6} it follows that the index of $q^{-1}(Q)$ at the singular point $p\in \Sigma$ with respect the cohomology class $a=[\frac{1}{2\pi}d\theta]\in H^1(\mathbb{S}^1)$, where $d\theta$ is the angular form on $\mathbb{S}^1$ is given by
\begin{equation}
ind_p(Q; a)=\sum_{i=1} ^{r}\frac{1}{2\pi k_i}\int_{\gamma_i} d\theta,
\label{eq:24}
\end{equation}
where $k_i$ is the number of elements of the orbit $O_i$, and $\gamma_i$ is the index of the point $p$ corresponding to the orbit $ O_i$.
Let us choose a frame $(e_1, e_2)$ along the curve $\gamma$, and we consider a unit vector field $X(t), 0\leq t\leq 1$ such that $\omega_{\gamma(t)}(X(t))$ around the curve $\gamma: I\to U'(p)$. If $\tilde{\theta}$ is the angle between $e_1$ and $X(0)$, we obtain that the index of $Q$ at the point $p$ with respect to the form $a$ can be also calculated in terms of this rotation angle by the formula
\begin{equation}
ind_p(Q, O_i, a)=\frac{\tilde{\theta}(2k_i)-\tilde{\theta}(0)}{2\pi k_i}.
\label{eq:25}
\end{equation}
Note that if the action of $\pi_1(U'(p))$ on $\pi^{-1}(p)$ is transitive, then the equation \eqref{eq:24} reduces to the following
\begin{equation}
ind_p(Q; a)=\frac{1}{4\pi n}\int_{\gamma} d\theta,
\label{eq:26}
\end{equation}
where $\gamma$ is the index of $p$ in $\pi^{-1}(p)$, and it is also true that
\begin{equation}
ind_p(Q, \pi^{-1}(p), a)=\frac{\tilde{\theta}(2k_i)-\tilde{\theta}(0)}{4n\pi}.
\label{eq:27}
\end{equation}
The equation \eqref{eq:27} coincides with the index of a binary differential $n$-form given in \cite{fukui2012}.

Now, we note that the index of $Q$ at a singular point $x$ seen as a singularity of the bundle $\pi: PTM\to M$ is twice the index of the same point as a singular point of the bundle $\pi\circ q: \mathbb{S}^1TM \to M$.
\end{ex}
\begin{rem}
This construction can be used to calculate an index of singular points of singular distributions over a two dimensional manifold $M$.
In [\cite{spivak_3_1999}, pages 218-223], the author gives another constructions of indexes of singular points of $1$-dimensional singular distributions and branched covering of two sheets defined by such a distributions.
\label{rem:3}
\end{rem}

\section{Resolution of a branched section}
\label{sec:3}

Let $M$ be a two-dimensional oriented manifold, and $\xi = \left\{ \pi_E : E \to M \right\}$ be a fiber bundle.
Let $\Sigma$ be a discrete subset of the manifold $M$.

\begin{df}
Let $Q$ be an $n$-sheeted branched section of the bundle $\xi$ with singularity set $\Sigma$, $M' = M \setminus \Sigma$, $E' = \pi^{-1}(M')$, and $Q' = Q \cap E'$.
A \emph{resolution of $Q$} is a map $\iota : S \to E$, where $S$ is an oriented two-dimensional manifold with boundary, such that
\begin{enumerate}
\item
$\iota(S) = Q$;
\item
$\pi = \pi_E \circ \iota : S \to M$ is surjective;
\item
the map $\iota$ is an embedding of $S' = S \setminus \partial S$ onto $Q'$.
\end{enumerate}
In case $M$ is compact, we assume $S$ to be compact, too.
\label{df:3}
\end{df}

\begin{rem}
From Definition~\ref{df:3} it follows that $\pi_E(Q) = M$ and $\pi_E(\partial S) = \Sigma$.
\label{rem:1}
\end{rem}

\begin{ex}
Let $M = \mathbb{R}^2$, $E = \mathbb{P}T \mathbb{R}^2$ and a branched section is the solution of the differential equation $x y dx^2 - (x^2 - y^2) dx dy - x y dy^2 = 0$.
As the discriminant of this equation is $(x^2 - y^2)^2 - 4 (x y)^2 = (x^2 + y^2)^2$, this differential equation is a binary differential equation (see Example~\ref{ex:2}).
This differential equation is represented in the form $(x dx + y dy)(y dx - x dy) = 0$, therefore its solution $Q$ consists of two $1$-dimensional distributions $L_1$ and $L_2$ on $\mathbb{R}^2$ given respectively by the equations $x dx + y dy = 0$ and $y dx - x dy = 0$. 
One can easily see that these equations determine sections with singularities $s_1$ and $s_2$ of the bundle $E$,  which admit resolutions (see \cite{arias_malakhaltsev2015}), call them $S_1$ and $S_2$, so the manifold $S_1 \sqcup S_2$ is a resolution of
the branched section $Q$.
\label{ex:5}
\end{ex}

\begin{ex}
Let $M = \mathbb{R}^2$, $E = PT \mathbb{R}^2$ and the branched section $Q$ is the solution of the binary differential equation
\begin{equation}
y dx^2 - 2 x dx dy - y dy^2 = 0.
\label{eq:7}
\end{equation}
The discriminant of equation~\eqref{eq:7} is $4(x^2 + y^2)$, therefore this equation has two real roots for all $(x, y)$ different from the origin, and at the origin all the coefficients vanish.
That is why, equation~\eqref{eq:7} is a binary differential equation (see Example~\ref{ex:2}).

The standard coordinates $(x, y)$ on $\mathbb{R}^2$ induce a trivialization of the bundle $\pi_E = E = PT\mathbb{R}^2 \to M = \mathbb{R}^2$, namely for the one-dimensional subspace $l \in PT_{(x, y)} \mathbb{R}^2$ spanned by a vector $p\partial_x + q\partial_y$, we assign the point $(x, y, [p : q]) \in \mathbb{R}^2 \times \mathbb{R}P^1$.
Thus, $PT\mathbb{R}^2 \cong \mathbb{R}^2 \times \mathbb{R}P^1$, and
\begin{equation}
Q = \left\{ (x, y, [p : q]) \in \mathbb{R}^2 \times \mathbb{R}P^1 \mid y p^2 - 2x p q - y q^2 = 0 \right\}
\label{eq:8}
\end{equation}
In this case
\begin{equation}
\Sigma = (0, 0), \quad Q' = \left\{ (x, y, [p : q]) \in Q \mid x^2 + y^2 > 0 \right\}, \quad M' = \mathbb{R}^2 \setminus \left\{ (0, 0) \right\}.
\label{eq:9}
\end{equation}
The projection $\pi_{PT\mathbb{R}^2} : PT\mathbb{R}^2 \to \mathbb{R}^2$ restricted to $Q'$ is a trivial (as a fiber bundle) double covering of $M'$.
Indeed, take the following open sets $U_1$ and $U_2$:
\begin{equation}
U_1 = M' \setminus (-\infty, 0) \times \left\{ 0 \right\}\text{ and } U_2 = M' \setminus (0, \infty) \times \left\{ 0 \right\}
\label{eq:10}
\end{equation}
It is clear that $M' = U_1 \cup U_2$. Also, at the points of $U_1$ we have $x + \sqrt{x^2 + y^2} > 0$, and at the points of $U_2$ we have $x - \sqrt{x^2 + y^2} > 0$.

Now let us take two sections of the bundle $\pi_{PT\mathbb{R}^2} : PT\mathbb{R}^2 \to \mathbb{R}^2$ defined on $M'$:
\begin{equation}
s_1: (x, y) \mapsto
\left\{
\begin{array}{l}
(x, y, [x + \sqrt{x^2 + y^2} : y]), \quad (x, y) \in U_1,
\\
(x, y, [-y : x - \sqrt{x^2 + y^2}]), \quad (x, y) \in U_2,
\end{array}
\right.
\label{eq:11}
\end{equation}
and
\begin{equation}
s_2: (x, y) \mapsto
\left\{
\begin{array}{l}
(x, y, [-y : x + \sqrt{x^2 + y^2}] ), \quad (x, y) \in U_1,
\\
(x, y, [x - \sqrt{x^2 + y^2} : y] ), \quad (x, y) \in U_2,
\end{array}
\right.
\label{eq:12}
\end{equation}
Note that over $U_1 \cap U_2$ there holds
\begin{equation}
[x + \sqrt{x^2 + y^2} : y] = [-y : x - \sqrt{x^2 + y^2}] \text{ and } [-y : x + \sqrt{x^2 + y^2}] = [x - \sqrt{x^2 + y^2} : y],
\label{eq:13}
\end{equation}
therefore the sections $s_1$ and $s_2$ are well defined.
One can easily prove that $s_i(M') \subset Q'$, $i = 1, 2$, and $s_1(M') \cap s_2(M') = \emptyset$.
Therefore $Q'$ is a trivial double covering of $M'$.

Now let us construct a resolution of the branched section $Q$.
Recall that $\mathbb{S}^1 = \left\{ (u, v) \mid u^2 + v^2 = 1 \right\}$, then
let take the diffeomorphism
\begin{equation}
f : \mathbb{S}^1 \to \mathbb{R}P^1, (u, v) \mapsto
\left\{
\begin{array}{l}
[u + 1 : v ], \quad u > -1,
\\[2pt]
[- v : u - 1], \quad u < 1,
\end{array}
\right.
\label{eq:14}
\end{equation}
and then the diffeomorphism $f$ ``rotated'' at the angle $\pi/2$ gives the diffeomorphism,
\begin{equation}
g : \mathbb{S}^1 \to \mathbb{R}P^1, (u, v) \mapsto
\left\{
\begin{array}{l}
[- v : u + 1], \quad u > -1,
\\[2pt]
[ u - 1 : v], \quad u < 1.
\end{array}
\right.
\label{eq:15}
\end{equation}

We take $S_1 = S_2 = \mathbb{R}_+ \times \mathbb{S}^1 = [0, \infty) \times \mathbb{S}^1$, and $S_1' = S_2' = (0, \infty)$.
We set $S = S_1 \sqcup S_2$, then $S' = S_1' \sqcup S_2'$.
Then $\iota : S \to \mathbb{R}^2 \times \mathbb{R}P^1$ is given by
\begin{equation}
\rest{\iota}{S_1}(r, (u, v)) = (ru, rv, f(u, v)),
\quad
\rest{\iota}{S_2}(r, (x, y)) = (ru, rv, g(u, v)).
\label{eq:16}
\end{equation}

One can easy see that $\rest{\iota}{S'_i} : S'_i \to Q'_i$, $i=1, 2$ is a diffeomorphism.
For example, any point $(x, y, [p : q]) \in V_{11}$, is the image of the point $(r, (u, v))$ under the map $\rest{\iota}{S_1}$, where
\begin{equation}
u = \frac{x}{\sqrt{x^2 + y^2}}, \quad v = \frac{y}{\sqrt{x^2 + y^2}}, \quad r = \sqrt{x^2 + y^2}.
\label{eq:17}
\end{equation}
\label{ex:6}
\end{ex}

\begin{ex}
As a generalization of Examples~\ref{ex:5} and~\ref{ex:6} one can take $n$ sections with singularities \cite{arias_malakhaltsev2015} of a bundle $\xi = \pi_E : E \to M$ which have the same set of singularities $\Sigma$, call them $s_i$, $i=\overline{1, n}$.
These sections define a branched section $Q$ of the bundle $\xi$: $Q = \left\{ s_i(x) \mid x \in M \setminus \Sigma \right\}$.
If $S_i$ is a resolution of $s_i$, then $S = \sqcup S_i$ is a resolution of $Q$.
\label{ex:7}
\end{ex}

\begin{ex}
Let us present an example of branched section, where the covering $\rest{\pi_Q}{Q'} : Q' \to M'$ is not trivial.
Take $M = \mathbb{R}^2 = \mathbb{C}$, $E = \mathbb{S}^1(\mathbb{C}) = \mathbb{C} \times \mathbb{S}^1$, the bundle of unit vectors over $M$, and let
\begin{equation}
Q = \left\{ (z, w) \in \mathbb{C} \times \mathbb{S}^1 \mid |z| w^2 = z\right\}.
\label{eq:18}
\end{equation}
Then $M' = \mathbb{C} \setminus \left\{ 0 \right\}$, $Q' = \left\{ (z, w) \mid w^2 = z/|z| \right\}$, and it is well known that $\rest{\pi_Q}{Q'} : Q' \to M'$ is a non trivial double covering.
Now let us take
\begin{equation}
S = [0, \infty) \times \mathbb{S}^1, \text{ and }
\iota: S \to E, \quad (r, e^{i\varphi}) \mapsto (r e^{2 i \varphi}, e^{i \varphi})
\label{eq:19}
\end{equation}
Then $S' = (0, \infty)$, and it is clear that the properties~(1)--(3) of Definition~\ref{df:3} hold true for $\iota$.
\label{ex:8}
\end{ex}

\begin{ex}
Let us present another example of branched section, where the covering $\rest{\pi_Q}{Q'} : Q' \to M'$ is not trivial.
Take $M = \mathbb{R}^2 = \mathbb{C}$, $E = PT\mathbb{R}^2 = \mathbb{R}^2 \times \mathbb{R}P^1 = \mathbb{C} \times \mathbb{R}P^1$, and let
\begin{equation}
Q = \left\{ (z, [ w ] \mid |w| = 1 \text{ and } |z|^2 w^4 = z^2\right\}
\label{eq:20}
\end{equation}
Then $M' = \mathbb{C} \setminus \left\{ 0 \right\}$, $Q' = \left\{ (z, w) \mid w^4 = z^2/|z|^2 \right\}$, and it is clear that $\rest{\pi_Q}{Q'} : Q' \to M'$ is a non trivial double covering.
Now let us take
\begin{equation}
S = [0, \infty) \times \mathbb{R}P^1, \text{ and }
\iota: S \to E, \quad (r, [ w ])\mapsto (r w^2, [ w ]),
\label{eq:21}
\end{equation}
where $|w|=1$.
 Then $S' = (0, \infty)$, and it is clear that the properties~(1)--(3) of Definition~\ref{df:3} hold true for $\iota$.
\label{ex:9}
\end{ex}

\begin{rem}
In Examples \ref{ex:8}--\ref{ex:9}, for each $x \in M$, the set $S_x$ is a discrete set if $x \in M \setminus \Sigma$, or is diffeomorphic to a circle $\mathbb{S}^1$ if $x \in \Sigma$.
\label{rem:2}
\end{rem}
Now let us consider a point $x \in \Sigma$.
Then, according to Definition~\ref{df:3}, $S_x = \pi^{-1}(x)$ consists of the connected components of the boundary $\partial S$.
Let us denote by $C(S_x)$ the set of connected components of $S_x$.
As $S_x$ is compact, the set $C(S_x)$ is finite, and each element of this set is diffeomorphic to a circle $\mathbb{S}^1$.

\begin{stat}
Let $C$ be a connected component of a boundary.
Then there exists a neighborhood $N(C)$ of $C$ and a diffeomorphism $f_C : N(C) \to \mathbb{S}^1 \times [0, 1]$ such that $f_C(C) = \mathbb{S}^1 \times \left\{ 0 \right\}$ and $U(x) = \pi(N(C)$ is a disk neighborhood of $x$.
For each $y \in U'(x)$, the set of orbits $\mathcal{O}_y$ consists of only one element.
In this cases the curve $\widehat \gamma$ corresponding to the orbit by Statement~\ref{stat:2} is a generator of the group $\pi_1(N(C))
\cong \mathbb{Z}$.
\label{stat:5}
\end{stat}
\begin{proof}
Indeed, $N(C) \setminus C$ is homeomorphic to a ring and $U'(x)$ is homeomorphic to a ring as well.
The map $N(C) \setminus C \to U'(x)$ induced by $\pi$ is a $n$-fold covering therefore $\pi_* : \pi_1(N(C)) \cong \mathbb{Z} \to \pi_1(U'(x))$ has the form $m \to km$.
At the same time $\pi_* ([\widetilde\gamma]) = \gamma^k$, thus $[\widetilde\gamma]$ is a generator of the group $\pi_1(N(C))$.
\end{proof}

\begin{cor}
The curve $\widetilde\gamma$ is homotopic in $N(C) \subset S$ to the curve $C \subset E_x$.
Therefore the curve $C$ represents $ind_x(Q, O)$.
\label{cor:2}
\end{cor}

\section{Connection and the Gauss-Bonnet theorem}
\label{sec:4}
Let $\xi = (\pi_E : E \to M)$ be a locally trivial fiber bundle with standard fiber $F$ and structure group $G$.
Assume that $G$ is a connected Lie group. 

Let $(U, \psi: \pi^{-1}(U) \to U \times F)$ be a chart of the atlas of $\xi$.
Let
\begin{equation}
\eta = p_F \circ \psi: \pi^{-1}(U) \to F,
\label{eq:6100}
\end{equation}
where $p_F : U \times F \to F$ is the canonical projection onto $F$.
For each $x \in U$ the map $\eta$ restricted to $F_x = \pi^{-1}(x)$ induces a diffeomorphism $\eta_x : F_x \to F$, and let $i_x : F \to F_x$ be the inverse of $\eta_x$.
Note that
if we take another chart $(U', \psi': \pi^{-1}(U') \to U' \times F)$, and $\eta' : \pi^{-1}(U') \to F$ is the corresponding map, then on $\pi^{-1}(U \bigcap U')$ we have that
\begin{equation}
\psi' \circ \psi^{-1} : (U \cap U') \times F \to (U \cap U') \times F, \quad (x, y) \mapsto (x, g(x) y),
\label{eq:7100}
\end{equation}
where $g : U \cap U' \to G$ is the gluing map of the charts.
Now, for any $x \in U \cap U'$, we have $\eta'_x \circ \eta_x^{-1} (y) = g(x) y$, and, as $G$ is connected, $\eta'_x \circ \eta_x^{-1} : F \to F$ is homotopic to the identity map.
This means that for any $x \in m$ we have well defined isomorphisms of the homotopy and (co)homology groups:
\begin{equation}
\begin{array}{l}
\pi_*(\eta_x) : \pi_*(F_x) \to \pi_*(F),
\\[5pt]
H_*(\eta_x) : H_*(F_x) \to H_*(F), \quad H^*(\eta_x) : H^*(F) \to H^*(F_x),
\end{array}
\label{eq:8100}
\end{equation}
which do not depend on the chart.

In \cite{arias_malakhaltsev2015}, for a locally trivial bundle  with standard fiber $F$ and  structure Lie group $G$, we have proved the following statement (\cite{arias_malakhaltsev2015}, Statement~1):
\begin{stat}
Let $a \in H^1(F)$ and $H$ be a connection in $E$.
There exists a $1$-form $\alpha \in \Omega^1(E)$ such that
\begin{enumerate}
\item
$\left.\alpha\right|_H = 0$;
\item
for each $x \in M$, $d i_x^* \alpha = 0$ and $[ i_x^* \alpha ] = H^1(\eta_x) a$.
\end{enumerate}
\label{stat:7}
\end{stat}
The decomposition $TE = H \oplus V$ gives a bicomplex representation of the complex $\Omega(E)$,
then the form $\alpha$ lies in $\Omega^{(0, 1)}(E)$ and
$d \alpha = \theta_{(1, 1)} + \theta_{(2, 0)}$, where $\theta_{(1, 1)} \in \Omega^{(1, 1)}$ and$\theta_{(2, 0)} \in \Omega^{(2, 0)}$, and
\begin{equation}
\theta_{(1, 1)}(X, Y) = (L_X \alpha)(Y), \quad \theta_{(2, 0)} = \widetilde{\alpha}(\Omega).
\label{eq:28}
\end{equation}
where $L_X$ is the Lie derivative with respect to the vector field $X$, and $\Omega$ is the curvature form of the connection $H$ (for details see \cite{arias_malakhaltsev2015}, Section~3).

Now let $Q$ be a branched section of the bundle $\xi$ which admits a resolution $\iota : S \to E$ (see Definition~\ref{df:3}).
Let us fix an element $a \in H^1(F)$, and let $\alpha \in \Omega^1(E)$ be the corresponding $1$-form (see Statement~\ref{stat:7}).
Then, by the Stokes theorem we have
\begin{equation}
\int_{\partial S} \iota^*\alpha = \int_S \iota^*d\alpha.
\label{eq:29}
\end{equation}

By Remark~\ref{rem:1} we have that $\pi_E(\partial S) = \Sigma$.
For $x \in \Sigma$, let $C(S_x)$ be the set of connected components of $\pi_E^{-1}(x)$.

From Corollary~\ref{cor:2}, it follows that, for $C \in C(S_x)$, we have
\begin{equation}
\int_C \alpha = \int_{\gamma(Q; y, O(C)} i_x^*\alpha,
\label{eq:30}
\end{equation}
where $\gamma(Q; y, O(C))$ represents the class $ind_x(Q; y, O(C)) \in \pi_1(F)$, and $O(C)$ is the orbit of the local monodromy group corresponding to $C$.
Therefore, from \eqref{eq:6} we have that
\begin{equation}
ind_x(Q; a) = \sum\limits_{C \in C(S_x)} \frac{1}{\# O(C)}\int_C \alpha.
\label{eq:31}
\end{equation}

If all the orbits of the local monodromy group corresponding to the components $C \in C(S_x)$ have the same number of elements $N(x)$, then
\begin{equation}
\int_{\partial S} \iota^*\alpha =\sum\limits_{x \in \Sigma} \sum\limits_{C \in C(S_x)} \int_C \alpha =
\sum\limits_{x \in \Sigma} N(x)\ ind_x(Q; a)
\label{eq:32}
\end{equation}

Thus we get the following theorem
\begin{thm}[Gauss-Bonnet-Hopf-Poincaré for branched sections]
If, for any $x \in \Sigma$, all the orbits of the local monodromy group corresponding to the components $C \in C(S_x)$ have the same number of elements $N(x)$, then
\begin{equation*}
\int_S \iota^*\theta_{(1, 1)} + \iota^*\theta_{(2, 0)} = \sum\limits_{x \in \pi(\partial S)} N(x)\ ind_x(Q).
\end{equation*}
\label{thm:1}
\end{thm}


\bigskip


\begin{thebibliography}{1}

\bibitem{arias_malakhaltsev2015}
Arias~F. A. and M.~A. Malakhaltsev.
\newblock A generalización of the Gauss-Bonnet and Hopf-Poincaré theorems.
\newblock {\em ArXiv:1510.01395 [MathDG] 5 Oct 2015}, 2015.

\bibitem{dubrovin_novikov_fomenko1985}
B.~A. Dubrovin, A.~T. Fomenko, and S.~P. Novikov.
\newblock {\em Modern Geometry- Methods and Applications Part II. The Geometry
and Topology of Manifolds}.
\newblock Springer-Verlag, 1985.

\bibitem{fukui2012}
T.~Fukui and J.~J. Nu{\~{n}}o-Ballesteros.
\newblock Isolated singularities of binary differential equations of degree
{\textdollar}n{\textdollar}.
\newblock {\em Publicacions Matem{\`{a}}tiques}, 56:65--89, jan 2012.

\bibitem{bruce1995}
J~W Bruce and F~Tari.
\newblock On binary differential equations.
\newblock {\em Nonlinearity}, 8(2):255--271, mar 1995.

\bibitem{spivak_3_1999}
Michael {Spivak}.
\newblock {\em {A comprehensive introduction to differential geometry. Vol.
1-5. 3rd ed. with corrections.}}, volume~3.
\newblock Houston, TX: Publish or Perish, 3rd ed. with corrections edition,
1999.
\end{thebibliography}

\end{document}